\documentclass[11pt]{amsart}

\usepackage[foot]{amsaddr}
\usepackage{color,graphicx}
\usepackage{dsfont}
\usepackage{amssymb}
\usepackage{amsmath}
\usepackage{amsfonts}
\usepackage{graphicx}
\usepackage{amsthm}
\usepackage{enumerate}
\usepackage[mathscr]{eucal}
\usepackage{mathrsfs}
\usepackage{verbatim}
\usepackage{yhmath}
\usepackage{pifont}

\calclayout

\usepackage{hyperref}
\hypersetup{colorlinks, linkcolor=blue}
\usepackage{soul}
\soulregister\cite7
\soulregister\eqref7
\soulregister\eqref7
\hypersetup{colorlinks, linkcolor=blue}

\newtheorem{thm}{Theorem}[section]
\newtheorem{lemma}[thm]{Lemma}
\newtheorem{prop}[thm]{Proposition}

\newtheorem{conj}[thm]{Conjecture}
\theoremstyle{definition}

\newtheorem{remark}[thm]{Remark}

\newcommand\ddd{\mathrm{d}}
\newcommand\dist{\mathrm{dist}}

\newcommand\bR{\mathbb{R}}

\newcommand\bC{\mathbb{C}}
\newcommand\bS{\mathbb{S}}

\def \l {\left}
\def \r {\right}

\makeatletter
\@namedef{subjclassname@2020}{
	\textup{2020} Mathematics Subject Classification}
\makeatother

\allowdisplaybreaks[4]

\begin{document}
\title[Strichartz-stability minimizer]{Stability for Strichartz inequalities: Existence of minimizers}

\author{Boning Di$^{1,2}$}
\author{Dunyan Yan$^3$}
\address{$^1$ School of Mathematical Sciences, Beijing University of Posts and Telecommunications}
\address{$^2$ Key Laboratory of Mathematics and Information Networks (Beijing University of Posts and Telecommunications), Ministry of Education}
\address{$^3$ School of Mathematical Sciences, University of Chinese Academy of Sciences}
\email{diboning@bupt.edu.cn}
\email{ydunyan@ucas.ac.cn}

\date{}
\thanks{This work was supported in part by the National Key R\&D Program of China [Grant No. 2023YFC3007303], National Natural Science Foundation of China [Grant Nos. 12071052 \& 12271501 \& 12471232 \& 12501130], and China Postdoctoral Science Foundation [Grant Nos. GZB20230812 \& 2024M753436 \& YJC20250892].}

\begin{abstract}
	We study the quantitative stability associated with the adjoint Fourier restriction inequality, focusing on the paraboloid and two-dimensional sphere cases. We show that these Strichartz-stability inequalities admit minimizers attaining their sharp constants, provided that these sharp constants are strictly smaller than the corresponding spectral-gap constants. Furthermore, for the two-dimensional sphere case, we obtain the existence of minimizers.
\end{abstract}

\maketitle

\tableofcontents

\section{Introduction}
The study of Fourier restriction phenomena plays a central role in harmonic analysis over the past decades. Roughly speaking, the well-known Paraboloid/Sphere Fourier Restriction Conjecture due to Stein asks the following question: for some given exponents $(p,q)$, do the associated $L^p\to L^q$ adjoint Fourier restriction inequalities hold? For the case $p=2$, this question was fully solved by Tomas-Stein \cite{Tomas1975}. In this paper, we focus on these Tomas-Stein inequalities for the paraboloid and sphere, which are also known as Strichartz inequalities \cite{Strichartz1977} since the space-time Fourier transform of PDE-solutions are supported on these surfaces.

For convenience, we introduce some necessary terminologies. Denoting the paraboloid/sphere surface $S \subset \bR^{d+1}$ and the associated adjoint Fourier operator $E_{S}$, then the Tomas-Stein inequalities state that
\begin{equation} \label{E:Strichartz General}
	\l\|E_Sf \r\|_{L^q(\bR^{d+1})} \leq C_{\sharp} \|f\|_{L^2(S)}, \quad q:=2+4/d, \quad C_{\sharp}:= \sup_{0\neq f\in L^2(S)} \frac{\l\|E_Sf \r\|_{L^q(\bR^{d+1})}}{\|f\|_{L^2(S)}}.
\end{equation}
Here the surfaces are equipped with their natural measure, and further details can be seen later. For these Strichartz inequalities, we introduce the associated \textit{Strichartz deficit} functional
\[\delta_{S}(f):= C_{\sharp}^2 \|f\|_{L^2(S)}^2 - \l\|E_Sf \r\|_{L^q(\bR^{d+1})};\]
and we call a nonzero function $f_0 \in L^2(S)$ a \textit{maximizer} for $C_{\sharp}$ if it satisfies $\delta_S(f_0) =0$; in addition, for a manifold $\mathcal{M} \subset L^2(S)$, we introduce the following manifold distance
\[\dist(f,\mathcal{M}):= \inf_{g\in \mathcal{M}} \l\|f-g\r\|_{L^2(S)}.\]

There are natural questions whether these Strichartz inequalities \eqref{E:Strichartz General} admit maximizers, and moreover, whether these Strichartz inequalities are stable in the sense that this Strichartz deficit functional $\delta_{S}(f)$ can control this distance from the maximizer manifold $\dist(f,\mathcal{M})$. In fact, the maximizers are conjectured to be Gaussian functions for the paraboloid case, and conjectured to be constant functions for the sphere case; while these two conjectures are only confirmed for the paraboloid with $d=\{1,2\}$ by Foschi \cite{Foschi2007} and Hundertmark-Zharnitsky \cite{HZ2006}, then confirmed for the two-dimensional sphere $\bS^2 \subset \bR^3$ by Foschi \cite{Foschi2015}. For more progress on this maximizer problems we refer to the survey papers \cite{FO2017,NOT2023,Oliveira2024}. Next, for the quantitative stability of Strichartz inequalities\footnote{Indeed, the precompactness of maximizing sequences can give a qualitative stability result, further details on this aspect are referred to \cite{BOQ2020,DY2024,CS2012A&P,FLS2016,Kunze2003,Ramos2012,Shao2009EJDE,Shao2016} and the references therein.}, Duyckaerts-Merle-Roudenko\cite{DMR2011} have shown the coercivity of a relevant quadratic form for the paraboloid with $d=\{1,2\}$; and recently, Gon\c{c}alves-Negro\cite{GN2022} have shown the following quantitative stability estimates for the paraboloid and sphere cases\footnote{They have also studied the wave equation case, which essentially corresponds to the cone case. See also the papers \cite{Goncalves2019,Goncalves2019JFA,Negro2023BLMS,Negro2023} for similar quantitative stability on Fourier restriction type inequalities. In a related direction, see \cite{CNO2024} for weighted-perturbation stability of maximizers in spherical extension inequalities.}: there exist $\eta\in (0,1)$ and $C_{\sharp\sharp}>0$ such that
\begin{equation} \label{E:Deficit general}
	\delta_{S}(f) \geq C_{\sharp\sharp} \dist(f,\mathcal{M})^2,
\end{equation}
under the condition that $\dist(f,\mathcal{M}) < \eta \|f\|_{L^2(S)}$; and this condition can be removed if the corresponding maximizer conjecture is confirmed. Similarly, we call a nonzero function $f_0 \in L^2(S)$ a \textit{minimizer} for $C_{\sharp\sharp}$ if it can make the inequality \eqref{E:Deficit general} an equality.

However, the results in \cite{GN2022} do not give too much information on the stability constants $C_{\sharp\sharp}$. On the one hand, what is the optimal value of this stability constant $C_{\sharp\sharp}$; on the other hand, whether this value $C_{\sharp\sharp}$ can be attained for a minimizer. As a comparison, we recall the corresponding research progress on the Sobolev inequality: the maximizers are exactly Aubin-Talenti manifold according to the celebrated works \cite{Aubin1976,Lieb1983,Rosen1971,Talenti1976}, then the Sobolev-stability questions are raised by Br\'ezis-Lieb \cite{BL1985} and the corresponding Sobolev stability estimate is established by Bianchi-Egnell \cite{BE1991}, recently the explicit stability constant is investigated by Dolbeault-Esteban-Figalli-Frank-Loss \cite{DEFFL2025} whose result is dimensionally sharp, and the existence of minimizers for stability constant is obtained by K\"onig \cite{Konig2025}. For further details on this aspect, we refer to the survey papers \cite{DEFFL2024,Frank2024} and the references therein, see also \cite{Christ2014,Christ2017,CLT2023,CLT2024,DSW2025,FMP2008,Fusco2015,FZ2022} for quantitative stability in different settings. Here we emphasize that in the Sobolev setting, by the work of \cite[Section 1.1]{Konig2025}, there are two crucial phenomena (two-peak and spectral-gap) which both play a role and thus lead to two compactness thresholds.

The main purpose of this paper is, based on the work \cite{GN2022} and motivated by the work \cite{Konig2025}, to \textbf{further investigate the quantitative Strichartz stability inequalities}. In this paper, we have established the following consequences: 
\noindent
\begin{itemize}
	\item for the sharp Strichartz stability constants  $C_{\sharp\sharp}$ in \eqref{E:Deficit general}, we compute the corresponding spectral-gap constants which directly give some upper bounds for $C_{\sharp\sharp}$;
	\item different from the Sobolev setting, for the existence of minimizers for $C_{\sharp\sharp}$, there is only one compactness threshold (the two-peak phenomenon is essentially vanished);
	\item for the existence of minimizers for $C_{\sharp\sharp}$, we discover an interesting fact that the sphere and paraboloid cases exhibit quite different behaviors: we obtain the existence of minimizers for the two-dimensional sphere, but only conditional existence of minimizers for the paraboloid.
\end{itemize}
Our main results are summarized as follows. Let $C_{*}$ and $C_{**}$ denote the sharp stability constants in \eqref{E:Deficit general} for the paraboloid and the two-dimensional sphere situation respectively, see also the precise expressions in the following \eqref{E:Strichartz stability-paraboloid} and \eqref{E:Strichartz stability-sphere}. Then we have
\begin{thm}[Sphere-stability minimizer] \label{T:Stability minimizer-2D sphere}
	There exists a minimizer for $C_{**}$.
\end{thm}
\begin{remark}
	In fact, our result implies that the minimizing sequence for $C_{**}$ cannot approach the maximizer manifold. Therefore, one cannot investigate the sharp constant $C_{**}$ by investigating the local behavior near the maximizer manifold.
\end{remark}
\begin{thm}[Paraboloid-stability minimizer] \label{T:Stability minimizer-paraboloid}
	There exists a minimizer for $C_{*}$, provided that
	\begin{equation} \label{T:Stability minimizer-paraboloid-1}
		C_{*} < \widetilde{C}_{d}, \quad \widetilde{C}_{d}:= \frac{d^2 +d +2}{(d+2)^3} 2^{\frac{2}{d+2}} \l(\frac{d}{d+2}\r)^{\frac{d^2}{2d+4}}.
	\end{equation}
	Note that the case $d\geq 3$ relies on the aforementioned Gaussian maximizer conjecture, which will be precisely stated later as Conjecture \ref{C:Gaussian maximizer}.
\end{thm}
\begin{remark}
	We indeed have shown that $C_{*}\leq \widetilde{C}_{d}$ must hold, and in fact, the equality $C_{*}= \widetilde{C}_{d}$ will imply the noncompactness\footnote{Note that this noncompactness does not imply the nonexistence of minimizers.} of minimizing sequences for $C_{*}$. We mention that, for this condition \eqref{T:Stability minimizer-paraboloid-1}, similar strict inequalities hold for the quantitative stability of Sobolev inequality \cite{Konig2023} and planar isoperimetric inequality \cite{BCH2017}, while their methods do not seem to work in our setting. 
\end{remark}

In the following two subsections, we present some precise terminologies on the paraboloid case and two-dimensional sphere case respectively. Meanwhile, we explain the main difficulties of our problem and the main methods of our proof. 
\subsection{Paraboloid adjoint Fourier restriction}
We recall the Fourier transform and Schr\"odinger operator
\[\widehat{f}(\xi):= \int_{\bR^d} e^{-2\pi ix \xi} f(x) \ddd x, \quad e^{it\Delta}f(x):= \int_{\bR^d} e^{2\pi i x\xi - 4\pi^2 it |\xi|^2} \widehat{f}(\xi) \ddd \xi.\]
Then the paraboloid situation Tomas-Stein/Strichartz inequality states that
\begin{equation} \label{E:Strichartz-paraboloid}
	\l\| e^{it\Delta} f\r\|_{L^q(\bR^{d+1})} \leq \mathbf{S}_d \|f\|_{L^2(\bR^d)}, \quad q:= 2+4/d, \quad \mathbf{S}_d:= \sup_{0\neq f\in L^2} \frac{\l\|e^{it\Delta} f\r\|_{L^q(\bR^{d+1})}}{\|f\|_{L^2(\bR^d)}}.
\end{equation}
We define the standard Gaussian function $G(x):= e^{-\pi |x|^2}$ and the Gaussian maximizer manifold
\[\mathcal{G}:= \l\{g: g=e^{it_0\Delta}\l[\lambda^{d/2} e^{2\pi i(x-x_0)\xi_0} e^{-\pi|\lambda(x-x_0)|^2}\r], \; (\lambda_0, \xi_0, x_0, t_0) \in \bR_{+} \times \bR^3\r\}.\]
\begin{conj}[Gaussian maximizer] \label{C:Gaussian maximizer}
	The maximizers for paraboloid Strichartz inequality \eqref{E:Strichartz-paraboloid} are exactly Gaussian-type functions $g\in \mathcal{G}$.
\end{conj}
As mentioned above, this conjecture has been confirmed for dimensions $d=\{1,2\}$ and remains open for higher dimensions $d\geq 3$. Thus there holds $\mathbf{S}_1= 12^{-1/12}$ and $\mathbf{S}_2=2^{-1/2}$. We define the following paraboloid Strichartz deficit functional and the Gaussian distance functional
\[\delta(f):= \mathbf{S}_d^2 \|f\|_{L^2(\bR^d)} - \l\|e^{it\Delta}f \r\|_{L^q(\bR^{d+1})}^2, \quad \dist(f,\mathcal{G}):= \inf_{g\in \mathcal{G}} \l\|f-g\r\|_{L^2(\bR^d)}.\]
The work of Gon\c{c}alves-Negro \cite[Theorem 1.2]{GN2022} shows that this deficit functional can control the Gaussian distance functional: more precisely, if the Gaussian maximizer Conjecture \ref{C:Gaussian maximizer} holds, then there exists a constant $C_{*}>0$ such that
\begin{equation} \label{E:Strichartz stability-paraboloid}
	\delta(f) \geq C_{*} \dist(f,\mathcal{G})^2, \quad C_{*}:= \inf_{0\neq f \in L^2} \frac{\delta(f)}{\dist(f,\mathcal{G})^2}.
\end{equation}
To establish the existence of minimizers for this paraboloid Strichartz stability inequality \eqref{E:Strichartz stability-paraboloid}, as stated in \cite{Konig2025}, we have the following two natural obstacles
\begin{itemize}
	\item \textbf{Spectral-gap:} the minimizing sequences might converge to the maximizer manifold.
	\item \textbf{Two-peak:} the minimizing sequences might consist of two maximizer profiles which are non-interacting in the limit sense.
\end{itemize}
Next we explain these two cases one by one.

First for the spectral-gap phenomenon, we recall some facts. To establish this quantitative stability result \eqref{E:Strichartz stability-paraboloid}, Gon\c{c}alves-Negro \cite[Section 2]{GN2022} has made use of the corresponding tangent space and its orthogonal complement space. Let $T_{G}\mathcal{G}$ be the tangent space of the manifold $\mathcal{G}$ at the Gaussian function $G$ and denote its orthogonal complement space as $T_{G}\mathcal{G}^{\perp}$. Then by direct calculation, see also \cite[Equation (4-1)]{GN2022}, there holds
\[T_{G}\mathcal{G}= \mathrm{span}_{\bC} \l\{e^{-\pi|x|^2}, x_1 e^{-\pi|x|^2}, x_2 e^{-\pi |x|^2}, \cdots, x_d e^{-\pi|x|^2}, |x|^2 e^{-\pi|x|^2}\r\}.\]
Here we define the following spectral-gap constant
\[C_{SG}:= \inf_{h\in T_{G}\mathcal{G}^{\perp}} \frac{\delta''(G)[h,h]}{2 \|h\|_2^2}, \quad \delta''(f)[h,h]:= \frac{\partial^2}{\partial \varepsilon^2}\Big|_{\varepsilon=0} \delta(f+\varepsilon h).\]
By this definition, if $h$ is a minimizer for $C_{SG}$ then one can conclude\footnote{Notice that we times one-half in the definition of spectral-gap constant to make sure this fact holds true. In addition, in the Proof of Proposition \ref{P:Paraboloid spectral-gap}, Step 3, we can show that this spectral gap constant $C_{SG}$ can be attained by taking the second order radial Hermite-Gaussian function.}
\[h\perp G \quad \Rightarrow \quad \delta(G+\varepsilon h)= C_{SG}\|h\|_{L^2(\bR^d)}^2 \varepsilon^2 + o(\varepsilon^2),\]
thus one can directly see that $C_{*}\leq C_{SG}$. Indeed, Gon\c{c}alves-Negro \cite[Proof of Theorem 1.2]{GN2022} has shown that $C_{SG}>0$. In this paper, we can further investigate some monotonicity to show the following Proposition \ref{P:Paraboloid spectral-gap}, which guarantees that the spectral-gap obstacle cannot appear due to the the condition \eqref{T:Stability minimizer-paraboloid-1}.
\begin{prop}[Paraboloid spectral-gap constant] \label{P:Paraboloid spectral-gap}
	The explicit value of the spectral-gap constant is
	\[C_{SG}= \widetilde{C}_d = \frac{d^2 +d +2}{(d+2)^3} 2^{\frac{2}{d+2}} \l(\frac{d}{d+2}\r)^{\frac{d^2}{2d+4}}.\]
\end{prop}
Second for the two-peak phenomenon, we define the paraboloid two-peak function $f_{\lambda}$ and two-peak constant $C_{TP}$ as follows
\[f_{\lambda}(x):= G(x)+ G_{\lambda}(x), \quad G_{\lambda}(x):= \lambda^{d/2}G(\lambda x),\quad C_{TP}:= \lim_{\lambda\to 0}\frac{\delta(f_{\lambda})}{\dist(f_{\lambda},\mathcal{G})^2}.\]
By these definitions one can directly see that $C_{*}\leq C_{TP}$. In fact, we can compute that
\begin{prop}[Paraboloid two-peak constant] \label{P:Paraboloid two-peak}
	The explicit value of the two-peak constant is
	\[C_{TP}= \l(2^{\frac{2}{d+2}} -1\r) \l(\frac{d}{d+2}\r)^{\frac{d^2}{2d+4}}.\]
\end{prop}
In summary, if there holds the strict inequality $C_{*}< \min\{C_{SG},C_{TP}\}$, then the spectral-gap and two-peak both cannot happen, and thus it is expected that the existence of minimizers can be established. In fact, by some direct computation, we can prove that the two-peak vanishes in our Strichartz setting, which comes from the following result.
\begin{prop}[Paraboloid two-peak vanishing] \label{P:Paraboloid two-peak vanishing}
	The spectral-gap constant and the two-peak constant satisfy the following relation
	\[C_{SG}<C_{TP}.\]
\end{prop}
\begin{remark}
	This conclusion reveals that, unlike the Sobolev setting studied by K\"onig in \cite{Konig2023,Konig2025} where both the two-peak and spectral-gap must be surmounted, the two-peak obstacle vanishes within our Strichartz setting, thereby reducing the problem solely to overcoming spectral-gap.
\end{remark}
\begin{proof}[\textbf{Proof of Proposition \ref{P:Paraboloid two-peak vanishing}}]
	The proof follows from direct computation. To show the desired conclusion, by extracting the common factor, it is enough to show
	\[2^{\frac{2}{d+2}} \frac{d^2+d+2}{(d+2)^3}>2^{\frac{2}{d+2}} -1 \quad \Longleftrightarrow \quad 1-2^{-\frac{2}{d+2}}> \frac{d^2 +d +2}{(d+2)^3}.\]
	Applying the inequality $1-e^{-x}\geq x-x^2/2$ with $x:= 2\ln 2/(d+2)$, we can obtain
	\[1-2^{-\frac{2}{d+2}} \geq \frac{2\ln 2}{d+2} - \frac{2(\ln 2)^2}{(d+2)^2}.\]
	Then we can further conclude that
	\begin{align*}
		1-2^{-\frac{2}{d+2}} - \frac{d^2 +d +2}{(d+2)^3} & \geq (d+2)^{-3} \l[2\ln 2(d+2)^2 -2(\ln 2)^2 (d+2) - d^2-d-2\r] \\
		&= (2\ln 2-1) d^2 +\l[8\ln2 -2(\ln 2)^2 -1\r] d + [8\ln 2 -4(\ln 2)^2 -2] >0.
	\end{align*}
	This implies the desired result and completes the proof.
\end{proof}

\subsection{Two-dimensional sphere adjoint Fourier restriction}
In order to be consistent with existing literature, for the spherical measure $\sigma$, we denote the spherical Fourier transform
\[\widehat{f\sigma}(x):= \int_{\bS^{d-1}} f(\theta) e^{-ix \theta} \ddd \sigma(\theta), \quad x\in \bR^d.\]
Then the sphere situation Tomas-Stein/Strichartz inequality states that
\begin{equation} \label{E:Strichartz sphere}
	\l\| \widehat{f\sigma} \r\|_{L^q(\bR^{d})} \leq \mathbf{M}_d \|f\|_{L^2(\bS^{d-1})}, \quad q:= 2+4/(d-1), \quad \mathbf{M}:= \sup_{0\neq f\in L^2} \frac{\l\|\widehat{f\sigma} \r\|_{L^q(\bR^3)}}{\|f\|_{L^2(\bS^{2})}}.
\end{equation}
We define the constant maximizer manifold
\[\mathcal{C}:= \l\{g: g=\lambda e^{iy\theta}, \; (\lambda,y) \in \bR_{+} \times \bR^3\r\}.\]
\begin{conj}[Constant maximizer] \label{C:Constant maximizer}
	The maximizers for sphere Strichartz inequality \eqref{E:Strichartz sphere} are exactly constant-type functions $g\in \mathcal{C}$.
\end{conj}
As mentioned above, this conjecture has been confirmed for dimensions $d=3$, and remains open for dimension $d=2$ and higher dimensions $d\geq 4$. Thus there holds $\mathbf{M} = 2\pi$. We define the following deficit functional of sphere-Strichartz inequality and the constant distance functional
\[\delta_{*}(f):= \mathbf{M}_d^2 \|f\|_{L^2(\bS^{d-1})} - \l\|\widehat{f\sigma} \r\|_{L^q(\bR^d)}^2, \quad \dist(f,\mathcal{C}):= \inf_{g\in \mathcal{C}} \l\|f-g\r\|_{L^2(\bS^{d-1})}.\]
The work of Gon\c{c}alves-Negro \cite[Theorem 1.3]{GN2022} shows that this deficit functional can control the constant distance functional: more precisely, for dimension $d=3$, then there exists a constant $C_{**}>0$ such that
\begin{equation} \label{E:Strichartz stability-sphere}
	\delta_{*}(f) \geq C_{**} \dist(f,\mathcal{C})^2,\quad C_{**}:= \inf_{0\neq f \in L^2} \frac{\delta_{*}(f)}{\dist(f,\mathcal{C})^2}.
\end{equation}
Similar to the paraboloid situation, there are two natural obstacles named spectral-gap and two-peak. Next we explain these two cases one by one.

First for the spectral-gap phenomenon, we recall some facts. Let $T_{1}\mathcal{C}$ be the tangent space of the manifold $\mathcal{C}$ at the constant function $1$ and denote its orthogonal space as $T_{1}\mathcal{C}^{\perp}$. Then, by some direct calculation, see also \cite[Page 1125]{GN2022}, there holds
\[T_1\mathcal{C}= \mathrm{span}_{\bR} \l\{1, i, ix_1, ix_2, ix_3\r\}.\]
Here we define the following spectral-gap constant
\[C_{SG*}:= \inf_{h\in T_{1}\mathcal{C}^{\perp}} \frac{\delta_{*}^{''}(1)[h,h]}{2 \|h\|_2^2}, \quad \delta_{*}^{''}(f)[h,h]:= \frac{\partial^2}{\partial \varepsilon^2}\Big|_{\varepsilon=0} \delta_{*}(f+\varepsilon h).\]
By this definition, if $h$ is a minimizer for $C_{SG*}$ then one can conclude\footnote{Notice that we times one-half in the definition of spectral-gap constant to make sure this fact holds true. And in the Proof of Proposition \ref{P:Sphere-spectral-gap}, Step 2, we can show that the spectral gap constant $C_{SG*}$ can be attained by taking the second order spherical harmonic function.}
\[h\perp 1 \quad \Rightarrow \quad \delta(1+\varepsilon h)= C_{SG*}\|h\|_{L^2(\bR^d)}^2 \varepsilon^2 + o(\varepsilon^2),\]
thus one can see that $C_{**}\leq C_{SG*}$. Indeed, Gon\c{c}alves-Negro \cite[Proof of Theorem 1.3]{GN2022} has shown that $C_{SG*}>0$. In this paper, by investigating some monotonicity in the work of Gon\c{c}alves-Negro and combining some spectral perturbative calculation due to K\"onig \cite{Konig2023}, we can show the following Proposition \ref{P:Sphere-spectral-gap}, which directly guarantees that the spectral-gap obstacle cannot appear.
\begin{prop}[Sphere spectral-gap constant] \label{P:Sphere-spectral-gap}
	The explicit value of the spectral-gap constant is
	\[C_{SG*}= \frac{8\pi^2}{5}.\]
	Furthermore, the stability-constant satisfies
	\[C_{**}< \frac{8\pi^2}{5}.\]
\end{prop}
Second for the two-peak phenomenon, we define the sphere two-peak function and sphere two-peak constant as follows
\[f_{y}(\theta):= 1+ e^{i\theta y}, \quad C_{TP*}:= \lim_{|y|\to \infty} \frac{\delta(f_{y})}{\dist(f_{y},\mathcal{C})^2}.\]
This two-peak constant is well-defined due to the rotation symmetry, and by these definitions one can directly see that $C_{**}\leq C_{TP*}$. In fact, we can compute that
\begin{prop}[Sphere two-peak constant] \label{P:Sphere-two-peak}
	The explicit value of the sphere-two-peak constant is
	\[C_{TP*}= \l(2-\sqrt{2}\r) \mathbf{M}^2 = (2-\sqrt{2})4\pi^2.\]
\end{prop}
Finally, it is obvious that the two-peak phenomenon vanishes since $C_{SG*}<C_{TP*}$, and moreover $C_{**}<\min\{C_{SG*},C_{TP*}\}$. Therefore the aforementioned two obstacles are eliminated, and we thus can obtain the desired unconditional existence of minimizers Theorem \ref{T:Stability minimizer-2D sphere}.

\subsection{Outline of the paper}
We show the paraboloid result Theorem \ref{T:Stability minimizer-paraboloid} in Section \ref{S:Conditional existence of minimizers-paraboloid}, and show the two-sphere result Theorem \ref{T:Stability minimizer-2D sphere} in Section \ref{S:Existence of minimizers-two-dimensional sphere}. In each section, we proceed three subsections to investigate the spectral-gap, two-peak, and existence of minimizers respectively.

\section{Conditional existence of minimizers: paraboloid} \label{S:Conditional existence of minimizers-paraboloid}
\subsection{Spectral-gap}
In this subsection, we compute the spectral-gap constant $C_{SG}$ in Proposition \ref{P:Paraboloid spectral-gap}, which comes from some more precise estimate based on \cite{GN2022}.

First, we recall that Gon\c{c}alves and Negro \cite[Pages 1119-1123]{GN2022} have shown the following results by using spherical harmonic functions and Lens transform: the minima of
\[f\mapsto \|f\|_2^{-2} \delta''(G)[f,f]\]
over all $T_{G}\mathcal{G}^{\perp}$ is the same as the minima over $T_{G}\mathcal{G}^{\perp} \cap L_{rad}^2(\bR^d)$; moreover for any function $f\in L_{rad}^2(\bR^d)$ with the following Laguerre polynomials expression
\begin{equation} \label{E:spectral-gap-0.5}
	f(x):=\sum_{m\geq 0} a(m) L_m^{d/2-1}(2\pi|x|^2) e^{-\pi |x|^2},
\end{equation}
where $L_m^{d/2-1}$ is the standard Laguerre polynomial, there holds
\begin{equation} \label{E:spectral-gap-0.7}
	\delta''(G)[f,f]= 2^{\frac{2-d}{2+d}} q^{-\frac{d^2}{2d+4}} \sum_{m\geq 2} \l[1-c_d(m)\r] |a(m)|^2 L_m^{d/2-1}(0),
\end{equation}
where the constant
\[c_d(m):= \frac{q}{2} \sum_{j=0}^{m} \binom{m+d/2-1}{m-j}\binom{m}{j} (1-2/q)^{2m-2j} (2/q)^{2j}.\]
Indeed, they have shown the stability of Strichartz inequality by showing that there exists a uniform constant $\varepsilon>0$ satisfies
\[c_d(m)<1-\varepsilon, \quad \forall m\geq 2.\]
Now, we are going to show that the function $c_d(m)$ decreases when $m\geq 2$, and then compute the value $c_d(2)$ to obtain the desired spectral-gap constant.
\begin{proof}[\textbf{Proof of Proposition \ref{P:Paraboloid spectral-gap}}]
	The proof proceeds in three steps. We first prove the monotonicity of the coefficients $c_d(m)$ in the initial two steps, and finally compute the desired spectral-gap constant.
	
	\textbf{Step 1:} decreasing of $c_d(m)$ for $d=\{1,2\}$. For the case $d=2$, one can direct compute that
	\[c_2(m)=\frac{2\cdot \binom{2m}{m}}{4^m} = \frac{2\cdot (2m)!}{4^m(m!)^2}, \quad \Rightarrow \quad \frac{c_2(m+1)}{c_2(m)} = \frac{2m+1}{2m+2} <1.\]
	For the case $d=1$, one can conclude that
	\[c_1(m)=3^{1-2m} S_m, \quad S_m:= \sum_{j=0}^m \binom{m}{j} \binom{m-1/2}{m-j}4^{m-j}.\]
	Recall that
	\[\binom{m}{j} = \frac{\Gamma(m+1)}{\Gamma(j+1)\Gamma(m-j+1)} = \frac{m!}{j!(m-j)!}\]
	and
	\[\binom{m-1/2}{m-j} =\frac{\Gamma(m+1/2)}{\Gamma(j+1/2) \Gamma(m-j+1)} = \frac{\frac{(2m)! \sqrt{\pi}}{4^m m!}}{\frac{(2j)! \sqrt{\pi}}{4^j j!} (m-j)!}\]
	Then we compute that
	\[S_m = \sum_{j=0}^m \frac{(2m)!}{(2j)! [(m-j)!]^2}.\]
	Next, by the trigonometric function integration formula
	\[\int_0^{\pi} \cos^{2k}\theta \ddd \theta = \pi \frac{(2k)!}{4^k(k!)^2},\]
	one can use binomial theorem to obtain the following integral representation
	\begin{align*}
		S_m &= \sum_{k=0}^m \frac{(2m)!}{(2m-2k)!} \frac{1}{(k!)^2} \\
		&= \sum_{k=0}^m \binom{2m}{2k} \frac{4^k}{\pi} \int_0^{\pi} \cos^{2k} \theta \ddd \theta \\
		&= \frac{1}{2\pi} \int_0^{\pi} \l[(2\cos\theta +1)^{2m} + (2\cos\theta -1)^{2m}\r] \ddd \theta.
	\end{align*}
	This directly gives the integral representation of $c_1(m)$ as follows
	\[c_1(m) = \frac{1}{2\pi} \int_0^{\pi} \l[\l(\frac{2\cos\theta +1}{3}\r)^{2m} + \l(\frac{2\cos\theta -1}{3}\r)^{2m}\r] \ddd \theta.\]
	Notice that $|2\cos\theta \pm 1|\leq 3$. Therefore, we conclude that $c_1(m)\geq c_1(m+1)$ for all $m\geq 2$, and thus the case $d=1$ is proved.
	
	\textbf{Step 2:} decreasing of $c_d(m)$ for $d\geq 3$. In the following we focus on the case $d\geq 3$. As shown in \cite[Page 1123]{GN2022}, by using the Jacobi polynomials
	\[P_m^{\alpha,\beta}(x):= 2^{-m} \sum_{k=0}^m \binom{m+\alpha}{m-k} \binom{m+\beta}{k} (x-1)^{m-k}(x+1)^k, \quad \alpha>-1, \beta>-1,\]
	one can obtain the following identity
	\[c_d(m)= (1+2/d)y^m P_m^{(0,d/2-1)}(x), \quad y:=\frac{d-2}{d+2}, \;\; x:= \frac{y+y^{-1}}{2} = \frac{d^2 +4}{d^2 -4}.\]
	Thus we investigate the following ratio
	\[\frac{c_d(m+1)}{c_d(m)} = y \frac{P_{m+1}^{(0,d/2-1)}(x)}{P_m^{(0,d/2-1)}(x)} = yr_m, \quad r_m:= \frac{P_{m+1}^{(0,d/2-1)}(x)}{P_m^{(0,d/2-1)}(x)}.\]
	Next we will prove the following result which direct implies the desired decreasing conclusion
	\begin{equation} \label{E:spectral-gap-1}
		r_m < y^{-1}, \quad \forall m\geq 2.
	\end{equation}
	By the 3-term recurrence relation for the Jacobi polynomials \cite[Section 4.5]{Szego1975} as follows
	\begin{align*}
		&2(n+1)(n+\alpha +\beta +1)(2n +\alpha +\beta) P_{n+1}^{(\alpha,\beta)}(x) \\
		&= (2n+\alpha+\beta +1) \l[(2n+\alpha+\beta +2)(2n +\alpha +\beta)x +\alpha^2 -\beta^2\r]P_n^{(\alpha,\beta)}(x) \\
		&\quad - 2(n+\alpha)(n+\beta)(2n+\alpha+beta +2) P_{n-1}^{(\alpha,\beta)}(x),
	\end{align*}
	we can take $\alpha=0$ and $\beta=d/2-1$ to conclude
	\begin{align*}
		r_m &= \frac{(2m+d/2)(2m+d/2+1)}{2(m+1)(m+d/2)} x + \frac{-(2m+d/2)(d/2-1)^2}{2(m+1)(m+d/2)(2m+d/2-1)} \\
		& \quad - \frac{m(m+d/2-1)(2m+d/2+1)}{(m+1)(m+d/2)(2m+d/2-1)} \frac{1}{r_{m-1}}.
	\end{align*}
	Here we denote this relation as follows
	\[r_m = a_m x + b_m - \frac{c_m}{r_{m-1}}, \quad a_m>0, \; b_m<0, \; c_m>0.\]
	Thus we investigate the function
	\[f_m(t):= a_m x +b_m -\frac{c_m}{t}, \quad t>0; \quad \Rightarrow \quad f_m'(t)= \frac{c_n}{t^2} >0.\]
	By this monotonicity, to prove the desired \eqref{E:spectral-gap-1}, it remains to show the following two facts
	\begin{equation} \label{E:spectral-gap-2}
		f_m(y^{-1}) < y^{-1}, \quad r_m<y^{-1}.
	\end{equation}
	First, recalling the definition of $x$ and $y$, we can compute that
	\[f_m(y^{-1}) -y^{-1} = -\frac{2(d^3 +4d^2m +16dm^2 +16dm-4d-16m)}{(d-2)^2 (d+2m)(m+1)(d+4m-2)} \frac{d-2}{d+2}<0, \quad \forall d\geq 3, \; \forall m\geq 2.\]
	Second, noticing that $r_0=y^{-1} -2/(d-2) <y^{-1}$ and then applying the induction argument together with the monotonicity of $f_m$, we can conclude that
	\[r_m =f_m(r_{m-1}) \leq f_m(y^{-1}) <y^{-1}.\]
	Hence we obtain the desired results \eqref{E:spectral-gap-2} and thus complete the proof of our first step.
	
	\textbf{Step 3:} the spectral-gap constant. We recall the radial Hermite–Gaussian functions
	\[\widetilde{L}_m(x):= L_m^{d/2-1}(2\pi |x|^2)e^{-\pi|x|^2}, \quad \Rightarrow \quad \langle \widetilde{L}_m, \widetilde{L}_n\rangle = 2^{-d/2} L_m^{d/2-1}(0) \delta_{mn},\]
	and they form an orthogonal basis of $L_{rad}^2(\bR^d)$. Here we have used the inner product notation $\langle f,g \rangle := \int f\bar{g} \ddd x$ and the Kronecker delta notation $\delta_{mn}$. Therefore, using the function in \eqref{E:spectral-gap-0.5} and the identity \eqref{E:spectral-gap-0.7}, we can see that
	\[\frac{\delta''(G)[f,f]}{\|f\|_2^2}= \frac{2^{\frac{2-d}{2+d}} q^{-\frac{d^2}{2d+4}} \sum_{m\geq 2} \l[1-c_d(m)\r] |a(m)|^2 L_m^{d/2-1}(0)}{2^{-d/2} \sum_{m\geq 0}|a(m)|^2 L_m^{d/2-1}(0)}.\]
	 In the previous steps we have shown that $c_d(m)$ decreases with respect to $m\geq 2$. Hence the weighted inequality gives that
	 \[\frac{\delta''(G)[f,f]}{\|f\|_2^2} \geq \frac{2^{\frac{2-d}{2+d}} q^{-\frac{d^2}{2d+4}} \l[1-c_d(2)\r]}{2^{-d/2}},\]
	 where the equality holds if and only if $f(x)=k \widetilde{L}_2(x)$. Finally, a direct computation shows that
	 \[c_d(x)= \frac{d^3 + 4d^2 +10d+4}{(d+2)^3}, \quad C_{SG}= \frac{d^2 +d +2}{(d+2)^3} 2^{\frac{2}{d+2}} \l(\frac{d}{d+2}\r)^{\frac{d^2}{2d+4}}.\]
	 This gives the desired explicit value and thus we complete the proof.
\end{proof}

\subsection{Two-peak}
In this subsection, we compute the two-peak constant $C_{TP}$ in Proposition \ref{P:Paraboloid two-peak}, which essentially comes from the profile decomposition conclusion, and this classical decomposition can be seen in the works such as \cite{BV2007,CK2007,DFY2025,DY2023,Keraani2001,MV1998}.
\begin{lemma}[Gaussian distance representation] \label{L:Gaussian distance representation}
	The Gaussian distance functional satisfies
	\[\dist(f,\mathcal{G})^2 = \|f\|_2^2 - m(f), \quad m(f):= \sup_{g\in \mathcal{G}_1} \l(\Re\int f \bar{g} \ddd x\r)^2, \quad \mathcal{G}_1:= \l\{g\in \mathcal{G}: \|g\|_2=1\r\}.\]
	And for each function $f\in L^2(\bR^d)$, its Gaussian distance $\dist(f,\mathcal{G})$ can be attained.
\end{lemma}
\begin{proof}[\textbf{Proof of Lemma \ref{L:Gaussian distance representation}}]
	The first part of this lemma follows from a direct expansion, and the second part can be shown by applying an approximate identity argument \cite[Section 1.2.4]{Grafakos2014a}. For the sake of brevity, we omit the detailed proof here. Further details can be found in \cite[Lemma 2.2]{Konig2025} and \cite[Lemma 3.3]{DEFFL2025}.
\end{proof}
\begin{proof}[\textbf{Proof of Proposition \ref{P:Paraboloid two-peak}}]
	By the definition of two-peak constant, this proof relies on the behavior of each terms as $\lambda \to 0$. We investigate these terms one by one.
	
	\textbf{Step 1:} we show that
	\begin{equation} \label{E:two-peak-1}
		\|f_{\lambda}\|_2^2 = 2^{1-d/2} +2 \lambda^{d/2} -d \lambda^{d/2+2} + o(\lambda^{d/2+2}) \quad \text{as} \;\; \lambda \to 0.
	\end{equation}
	Indeed, by direct calculation one can obtain
	\begin{align*}
		\|f_{\lambda}\|_2^2 &= \|G\|_2^2 + \|G_{\lambda}\|_2^2 + 2 \int_{\bR^d} G(x) G_{\lambda}(x) \ddd x \\
		&= 2^{1-d/2} + 2\lambda^{d/2} (1+\lambda^2)^{-d/2} \\
		&= 2^{1-d/2} + 2\lambda^{d/2} -d \lambda^{d/2+2} + o(\lambda^{d/2+2}).
	\end{align*}
	This gives the limit behavior of $\|f_{\lambda}\|_2^2$ and completes the first step.
	
	\textbf{Step 2:} we show that
	\begin{equation} \label{E:two-peak-2}
		\|e^{it\Delta}f_{\lambda}\|_q^q = \frac{1}{2}q^{-d/2} +o(1) \quad \text{as} \;\; \lambda \to 0.
	\end{equation}
	First, one can compute that
	\[e^{it\Delta} G_{\lambda}(x)= \frac{\lambda^{d/2}}{(1+4\pi it \lambda^2)^{d/2}} \exp \l(-\frac{\pi \lambda^2 |x|^2}{1+4\pi it \lambda^2}\r),\]
	and
	\[\l|e^{it\Delta} G(x)\r| = \frac{1}{(1+16\pi^2 t^2)^{d/4}} \exp \l(-\frac{\pi|x|^2}{1+16\pi^2 t^2}\r).\]
	Integrating with $x$ and then $t$, these further give
	\[\int_{\bR^d} \l|e^{it\Delta}G\r|^q \ddd x = q^{-d/2} (1+16\pi^2 t^2)^{-1},\]
	and
	\[\int_{\bR} q^{-d/2} (1+16\pi^2 t^2)^{-1} \ddd t = \frac{1}{4} q^{-d/2} = \frac{1}{4(2+4/d)^{d/2}}.\]
	By the scaling symmetry, we can see that
	\[\l\|e^{it\Delta}G_{\lambda}\r\|_q^q = \frac{1}{4(2+4/d)^{d/2}}.\]
	Next, to show the desired result \eqref{E:two-peak-2}, we estimate the following error term
	\[R(\lambda) := \|e^{it\Delta}f_{\lambda}\|_q^q - \|e^{it\Delta}G\|_q^q - \|e^{it\Delta}G_{\lambda}\|_q^q,\]
	which satisfies
	\[R(\lambda) \lesssim \int |e^{it\Delta}G|^{q-1}|e^{it\Delta}G_{\lambda}| \ddd x \ddd t + \int |e^{it\Delta}G|^{q-1}|e^{it\Delta}G_{\lambda}| \ddd x \ddd t =: I_1(\lambda) + I_2(\lambda).\]
	We are going to show that
	\[\lim_{\lambda \to0} I_1(\lambda) =0, \quad \lim_{\lambda \to0} I_2(\lambda) =0.\]
	Indeed, by the stationary phase method, one can see the following dispersive estimate
	\[\|e^{it\Delta}f\|_{L_x^\infty} \lesssim |t|^{-d/2} \|f\|_{L_x^1},\]
	which, by the interpolation with $L^2$-conservation law, further implies that
	\[\|e^{it\Delta}G\|_{L_x^q} \lesssim \langle t \rangle^{-\frac{d}{d+2}}, \quad \|e^{it\Delta}G_{\lambda} \|_{L_x^q} \lesssim \lambda^{\frac{d}{d+2}} \langle \lambda^2 t \rangle^{-\frac{d}{d+2}}, \quad \langle\cdot\rangle := (1+|\cdot|^2)^{1/2}.\]
	Hence, by H\"older's inequality, we obtain
	\[I_1(\lambda) \leq \int \|e^{it\Delta} G\|_{L_x^q}^{q-1} \|e^{it\Delta}G_{\lambda} \|_{L_x^q} \ddd t \lesssim \lambda^{\frac{d}{d+2}} \int_{\bR} \langle t \rangle^{-\frac{d+4}{d+2}} \langle \lambda^2 t \rangle^{-\frac{d}{d+2}} \ddd t.\]
	Then we divide the time region into two parts: if $t\leq \lambda^{-2}$, then $\langle \lambda^2 t \rangle \sim 1$ and
	\[\int_{|t|\leq \lambda^2} \langle t \rangle^{-\frac{d+4}{d+2}} \langle \lambda^2 t \rangle^{-\frac{d}{d+2}} \ddd t \sim \int_{|t|\leq \lambda^2} \langle t \rangle^{-\frac{d+4}{d+2}} \ddd t \lesssim 1;\]
	if $t\geq \lambda^{-2}$, then $\langle \lambda^2 t \rangle \sim \lambda^2 |t|$ and
	\[\int_{|t|\geq \lambda^{-2}} \langle t \rangle^{-\frac{d+4}{d+2}} \langle \lambda^2 t \rangle^{-\frac{d}{d+2}} \ddd t \lesssim \lambda^{-\frac{2d}{d+2}} \int_{|t|\geq \lambda^{-2}} |t|^{-2} \ddd t \lesssim \lambda^{\frac{4}{d+2}}.\]
	In summary, we conclude
	\[I_1(\lambda) \lesssim \lambda^{\frac{d}{d+2}};\]
	and similarly we can deduce the following decay estimate
	\[I_2(\lambda) \lesssim \lambda^{\frac{d}{d+2}}.\]
	These give the desired conclusion $R(\lambda) \to 0$ as $\lambda\to 0$, and thus we obtain the desired result \eqref{E:two-peak-2}.
	
	\textbf{Step 3:} we show that
	\begin{equation} \label{E:two-peak-3}
		\dist(f_{\lambda},\mathcal{G}) = 2^{-d/2} + o(\lambda^{d/2}) \quad \text{as} \;\; \lambda \to 0.
	\end{equation}
	By the Gaussian distance representation Lemma \ref{L:Gaussian distance representation}, we investigate the item $m(f_{\lambda})$. Since $f_{\lambda}$ is real-valued radial decreasing function, by writing $g\in\mathcal{G}_1$ as $g=g_1 +ig_2$ and the inequality
	\[\frac{\l(\Re\int f_{\lambda} \bar{g} \ddd x\r)^2}{\int |g_1|^2 + |g_2|^2 \ddd x} \leq \frac{(\int f_{\lambda} g_1 \ddd x)^2}{\|g_1\|_2^2},\]
	one can use rearrangement inequality \cite[Theorem 3.4]{LL2001} to see that
	\[m(f_{\lambda}) = \sup_{g\in \mathcal{G}_1} \l(\Re\int f_{\lambda} \bar{g} \ddd x\r)^2 = \sup_{\mu>0} \l(\int f_{\lambda} G_{\mu} \ddd x\r)^2\Big/\|G\|_2^2 =2^{-d/2} \sup_{\mu>0} \l(\int f_{\lambda} G_{\mu} \ddd x\r)^2.\]
	Denote $H_{\lambda}(\mu):= \int f_{\lambda} G_{\mu} \ddd x$. By direct computations, one can obtain
	\[H_{\lambda}(\mu)= \langle G, G_{\mu} \rangle + \langle G_{\lambda}, G_{\mu}\rangle = \l(\frac{\mu}{1+\mu^2}\r)^{d/2} + \l(\frac{\lambda \mu}{\lambda^2 +\mu^2}\r)^{d/2};\]
	then conclude the following inequality
	\[H_{\lambda}(\mu) \leq \|G+G_{\lambda}\|_2 \|G_{\mu}\|_2 \leq \l(\|G\|_2 +\|G_{\lambda}\|_2\r) \|G\|_2 = 2^{1-d/2},\]
	where the equality holds if and only if $\lambda =\mu =1$. On the other hand, there holds the identity $H_{\lambda}(\mu) = H_{\lambda}(\mu^{-1} \lambda)$, which implies that we only need to consider the case $\mu\in[\sqrt{\lambda},\infty)$ as we consider the value of $m(f_{\lambda})$. Now for each fixed $\lambda \in (0,1/2)$, we investigate
	\[\sup_{\mu\in[\sqrt{\lambda}, \infty)} H_{\lambda}(\mu) = \sup_{\mu\in[\sqrt{\lambda}, \infty)} \l[\l(\frac{\mu}{1+\mu^2}\r)^{d/2} + \l(\frac{\lambda \mu}{\lambda^2 +\mu^2}\r)^{d/2}\r] =: \sup_{\mu\in[\sqrt{\lambda}, \infty)} \l[S(\mu) + T(\lambda,\mu)\r].\]
	For each $\lambda$, due to the limit $\lim_{\mu \to \infty} H_{\lambda}(\mu) =0$, we see that the supremum of $\mu$ can be attained, which we denote as $\mu(\lambda)$. We claim that
	\begin{equation} \label{E:two-peak-4}
		\mu(\lambda) = 1+ 2^{d/2} \lambda^{d/2} +o(\lambda^{d/2}).
	\end{equation}
	Then using this claim, we conclude
	\begin{align*}
		m(f)^{1/2} &= 2^{d/4}\l[S(1) + S'(1)(\mu(\lambda)-1) +o(\mu(\lambda)-1) + T(\lambda,1) + o(T(\lambda,1))\r] \\
		&= 2^{d/4}\l[2^{-d/2} +  \l(\frac{\lambda}{\lambda^2+1}\r)^{d/2} +o(\lambda^{d/2})\r] \\
		&= 2^{-d/4} + 2^{d/4} \lambda^{d/2} + o(\lambda^{d/2}).
	\end{align*}
	Hence by Lemma \ref{L:Gaussian distance representation} and the conclusion of our first step, we obtain the desired conclusion
	\[\dist(f,\mathcal{G})^2= 2^{1-d/2} + 2\lambda^{d/2} -\l[2^{-d/4} + 2^{d/4} \lambda^{d/2} + o(\lambda^{d/2})\r]^2 + o(\lambda^{d/2}) = 2^{-d/2} +o(\lambda^{d/2}).\]
	It remains to prove the previous claim \eqref{E:two-peak-4}, which relies on some uniform analysis and implicit function theorem. By direct computation one can see
	\[S'(\mu)=\frac{d}{2}\frac{1-\mu^2}{(1+\mu^2)^2} \l(\frac{\mu}{1+\mu^2}\r)^{\frac{d}{2}-1}, \quad S''(1)= -\frac{d}{2^{1+d/2}} <0,\]
	and
	\[H_{\lambda}^{'}(\mu) =\frac{d}{2} \frac{1-\mu^2}{(1+\mu^2)^2} \l(\frac{\mu}{1+\mu^2}\r)^{\frac{d}{2}-1} + \frac{d}{2} \frac{\lambda(\lambda^2 -\mu^2)}{(\lambda^2 +\mu^2)^2} \l(\frac{\lambda \mu}{\lambda^2 + \mu^2}\r)^{\frac{d}{2}-1}.\]
	These imply that $H_{\lambda}(\mu)$ decreases on $\mu\in(1,\infty)$ and thus $\mu(\lambda) \in [\sqrt{\lambda},1]$. On the other hand, for $\mu\in[\sqrt{\lambda},\infty)$ we have the uniform convergence estimate
	\[T(\lambda,\mu) \leq \l(\frac{\lambda \mu}{\mu^2}\r)^{d/2} \leq \l(\frac{\lambda}{\mu}\r)^{d/2} \leq \lambda^{d/4} \to 0 \quad \text{as} \;\; \lambda \to 0.\]
	We then show that for arbitrary $\varepsilon>0$ small enough, there exists $\eta$ such that if $\lambda\in (0,\eta)$ then $\mu(\lambda) \in (1-\varepsilon,1+\varepsilon)$. For arbitrary $\varepsilon>0$, since the function $S$ has unique max-value on $\mu=1$, we define
	\[\eta_1:= S(1) -\max\l\{\sup_{\mu \leq 1-\varepsilon} S(\mu), \sup_{\mu\geq 1+\varepsilon} S(\mu)\r\}, \quad \eta:= \min \{\eta_1, (\eta_1/4)^{4/d}\},\]
	then for all $\mu\notin (1-\varepsilon,1+\varepsilon)$, we have
	\[H_{\lambda}(\mu) \leq S(\mu) + \frac{\eta}{4} \leq S(1)- \frac{3\eta}{4}.\]
	The continuity of $S$ implies that there exists $\mu_0 \in (1-\varepsilon,1)$ such that
	\[H_{\lambda}(\mu_0) \geq S(1) -\frac{\eta}{4},\]
	which further implies $\mu(\lambda) \in (1-\varepsilon,1+\varepsilon)$. This shows that $\lim_{\lambda \to0}\mu(\lambda) =1$. Next, we apply implicit function theorem to show the desired claim \eqref{E:two-peak-4}. Notice that
	\[F(\lambda,\mu):= H_{\lambda}^{'}(\mu), \quad \Rightarrow \quad F(0,1)=S'(1)=0, \;\; \frac{\partial}{\partial \mu}F(0,1) =S''(1) =-\frac{d}{2^{1+d/2}} \neq 0.\]
	Thus there exists unique function $\mu(\lambda)$ such that
	\[F(\lambda,\mu(\lambda)) =0, \quad \mu(0)=1.\]
	When $\lambda \to0$ and $\mu \sim 1$, we can use Taylor's expansion to deduce
	\[T(\lambda,\mu)= \l(\frac{\lambda}{\mu}\r)^{\frac{d}{2}} \l[1+O(\lambda^2/\mu^2)\r]^{-\frac{d}{2}} = \l(\frac{\lambda}{\mu}\r)^{\frac{d}{2}} + O(\lambda^{\frac{d}{2}+2}),\]
	and
	\[\frac{\partial}{\partial \mu} T(\lambda,\mu) = -\frac{d}{2}\lambda^{\frac{d}{2}} \mu^{-\frac{d}{2}-1} + O(\lambda^{\frac{d}{2} +2}),\]
	as well as
	\[ S'(\mu) =S''(\mu)(\mu-1) +O(|\mu-1|^2) = -\frac{d}{2^{1+d/2}} (\mu-1) +O(|\mu-1|^2).\]
	Taking these expansion into the aforementioned equation $F(\lambda,\mu(\lambda))=0$, we obtain
	\[-\frac{d}{2^{1+d/2}} [\mu(\lambda)-1] +O(|\mu(\lambda)-1|^2) -\frac{d}{2}\lambda^{\frac{d}{2}} \mu(\lambda)^{-\frac{d}{2}-1} +O(\lambda^{\frac{d}{2}+2}) =0,\]
	which directly gives the desired claim \eqref{E:two-peak-4} by the fact that $\lim_{\lambda\to 0}\mu(\lambda)=1$.
	
	\textbf{Step 4:} with the previous three steps finished, by taking the conclusions \eqref{E:two-peak-1} and \eqref{E:two-peak-2} and \eqref{E:two-peak-3} into the following expression, we obtain
	\begin{align*}
		\frac{\mathbf{S}_d^2 \|f_{\lambda}\|_2^2 - \|e^{it\Delta}f_{\lambda}\|_q^2}{\dist(f_{\lambda},\mathcal{G})^2} &=\frac{2^{-d/2}\l(2^{1-d/2} +2 \lambda^{d/2}\r) -2^{-2/q} q^{-d/q} +o(1)}{2^{-d/2} +o(\lambda^{d/2})} \\
		&= \l(2^{\frac{2}{d+2}} -1\r)\l(\frac{d}{d+2}\r)^{\frac{d^2}{2d+4}} +o(1).
	\end{align*}
	This implies the desired explicit value and thus completes the proof.
\end{proof}

\subsection{Conditional compactness of minimizing sequences} \label{SubS:Conditional compactness-paraboloid}
In this subsection, based on the assumption \eqref{T:Stability minimizer-paraboloid-1} and spectral-gap Proposition \ref{P:Paraboloid spectral-gap} together with two-peak vanishing Proposition \ref{P:Paraboloid two-peak vanishing}, we show the desired existence of minimizers for the paraboloid stability constant $C_{*}$.
\begin{proof}[\textbf{Proof of Theorem \ref{T:Stability minimizer-paraboloid}}]
	The proof strategy is very similar to the proof of \cite[Theorem 1.2]{Konig2025}. The arguments are pretty long, but standard, as shown in \cite[Section 4]{Konig2025}. Here for simplicity, we briefly summarize the proof steps and omit the details.
	\begin{itemize}
		\item[\ding{172}] For the minimizing sequence $f_n$, by the fact
		\[\mathbf{S}_d^2 -C_{SG} = 2^{-\frac{d}{d+2}} \l(\frac{d}{d+2}\r)^{\frac{d^2}{2d+4}} \frac{d^3 +4d^2 +10d+4}{(d+2)^3} >0,\]
		one can deduce $\|e^{it\Delta}f_n\|_q^2 \geq c_0 >0$ and then use the refined Strichartz estimate to obtain
		\[f_n \rightharpoonup f_0 \;\; \text{in} \;\; L^2(\bR^d), \quad f_0\neq 0.\]
		Then denoting $g_n:= f_n-f_0 \rightharpoonup 0$, we obtain the profile decomposition conclusions
		\[\|f_n\|_2^2 = \|f_0\|_2^2 + \|g_n\|_2^2 +o(1), \quad \|e^{it\Delta}f_n\|_q^q = \|e^{it\Delta} f_0\|_q^q + \|e^{it\Delta}g_n\|_q^q + o(1).\]
		\item[\ding{173}] On the one hand, by investigating the infinity of parameters and using the fact that $m(f)$ can be attained for each $f$, we have
		\[m(f_n)= \max\{m(f_0), m(g_n)\} +o(1).\]
		\item[\ding{174}] On the other hand, by using the minimizing sequence property (profile decomposition conclusions) and divide-two-part arguments, as well as scaling twice to reach mean-value and then investigating the monotonicity to deduce a contradiction to the sharp constant $C_*$, we can obtain
		\[m(f_0)= m(g_n) +o(1).\]
		\item[\ding{175}] Therefore, by scaling and the divide-two-part arguments, as well as investigating monotonicity and the fact $C_*\leq C_{SG}<C_{TP}$ due to Proposition \ref{P:Paraboloid two-peak vanishing}, which deduces a contradiction to the sharp constant $\mathbf{S}_d$, we conclude
		\[g_n \to 0.\]
		\item[\ding{176}] Finally, by the strong convergence of $g_n$, it remains to show that the limit-denominator is non-zero which means $f_0\notin \mathcal{G}$. This indeed comes from the local asymptotic analysis \cite[Theorem 2.1]{GN2022} and the assumption $C_* <C_{SG}$ due to Proposition \ref{P:Paraboloid spectral-gap}, which means that the spectral-gap phenomenon cannot happen.
	\end{itemize}
	Here we have used the limit-version of \cite[Theorem 2.1]{GN2022} which states that: if $\delta'(G) =0$ and there exists $\rho>0$ such that
	\[\delta''(G)[h,h] \geq \rho \|h\|_2^2, \quad \forall f\in T_G\mathcal{G}^{\perp},\]
	then for the sequence $f_n$ with $\dist(f_n,\mathcal{G}) \to 0$ one can obtain
	\[\dist(f_n,\mathcal{G})^{-2} \delta(f) \geq \rho/2.\]
	Finally, this outline finishes the proof.
\end{proof}

\section{Existence of minimizers: two-dimensional sphere} \label{S:Existence of minimizers-two-dimensional sphere}
\subsection{Spectral-gap}
In this subsection, we compute the spectral-gap constant $C_{SG*}$ and its relation with the sphere Strichartz stability constant $C_{**}$ as stated in Proposition \ref{P:Sphere-spectral-gap}, which comes from some more precise estimate based on \cite{GN2022}.

First, we recall that Gon\c{c}alves and Negro \cite[Pages 1125-1126]{GN2022} have shown the following results by using spherical harmonic functions and Bessel functions: for any function $f\in L^2(\bS^2)$ with expression
\begin{equation} \label{E:Sphere-spectral-gap-1}
	f(\theta):=\sum_{k\geq 2} a(k) Y_k(\theta) \in (T_1\mathcal{C})^{\perp},
\end{equation}
where $Y_k$ is the standard real spherical harmonics with $\|Y_k\|_{L^2(\bS^2)} =1$, there holds
\begin{equation} \label{E:Sphere-spectral-gap-2}
	\delta_{*}^{''}(1)[f,f]= 2\mathbf{M}^2 \sum_{k\geq 2} |a_k|^2 -c_0^{-1} \mathbf{M}^2 \sum_{k\geq 2} c_k\l[4|a_k|^2 + 2(-1)^k \mathrm{Re}(a_k)^2\r],
\end{equation}
where the constant
\[c_k:= \int_0^{\infty} |J_{\frac{1}{2}}(r)|^2 J_{\frac{1}{2}+k}(r)^2 \ddd r,\]
and $J_k$ is the Bessel function of the first kind. For further properties on Bessel functions and spherical harmonic functions, we refer to the chapters \cite[Appendix B]{Grafakos2014a} and \cite[Chapter 4]{SW1971}, see also the classical books \cite{Watson1944} and \cite{Muller1966}.
\begin{proof}[\textbf{Proof of Proposition \ref{P:Sphere-spectral-gap}}]
	The proof proceeds in three steps. We prove the decreasing of the constant $c_k$ in the first step, and then we compute the spectral-gap constant in the second step, finally we show that the stability constant is strictly smaller than this spectral-gap constant in the third step.
	
	\textbf{Step 1:} decreasing of $c_k$. We recall the spherical Bessel function $j_k(r) = \sqrt{\frac{\pi}{2r}} J_{k+\frac{1}{2}}(r)$ and the following relation
	\begin{equation} \label{E:Sphere-spectral-gap-3}
		j_{k+1}(r) = \frac{2k+1}{r} j_k(r) -j_{k-1}(r),
	\end{equation}
	with
	\[j_0(r)=\frac{\sin r}{r}, \quad j_1(r)= \frac{\sin r}{r^2} -\frac{\cos r}{r}, \quad j_2(r)=\l(\frac{3}{r^3} -\frac{1}{r}\r)\sin r -\frac{3}{r^2} \cos r.\]
	Hence we obtain
	\[c_k= \frac{4}{\pi^2} \int_0^{\infty} \sin^2r j_k^2(r) \ddd r = \frac{2}{\pi^2} \int_0^{\infty} j_k^2(r) \ddd r + \frac{2}{\pi^2} \int_0^{\infty} \cos(2r) j_k^2(r) \ddd r =: A_k +B_k.\]
	For the item $A_k$, by the Weber–Schafheitlin integral formula \cite[Page 405]{Watson1944} as follows
	\[\int_0^{\infty} t^{-1} J_v^2(at) \ddd t= \frac{1}{2v}, \quad \mathrm{Re} v>0,\]
	we obtain
	\[A_k= \frac{1}{\pi} \int_0^{\infty} \frac{J_{k+1/2}^2(r)}{r} \ddd r = \frac{1}{(2k+1)\pi}.\]
	For the item $B_k$, we compute that
	\[B_k= \frac{\pi}{2} \int_0^{\infty} \frac{\cos(2r)}{r} J_{k+1/2}^2(r) \ddd r,\]
	and then we introduce the notation
	\[B_{k,\eta}:= \frac{\pi}{2} \int_0^{\infty} e^{-\eta r^2} \frac{\cos(2r)}{r} J_{k+1/2}^2(r) \ddd r, \quad \eta >0.\]
	By the asymptotic expansion of Bessel function \cite[Page 199]{Watson1944}, for $r\gg 1$, we obtain
	\[J_{k+1/2}(r) = \sqrt{\frac{2}{\pi r}} \cos\l(r-\frac{(k+1/2)\pi}{2} -\frac{\pi}{4}\r) +O(r^{-3/2}), \quad r\to \infty.\]
	Thus, using the dominated convergence theorem, we conclude
	\[\frac{\cos(2r)}{r} J_{k+1/2}^2(r) \in L^1(0,\infty), \quad \Rightarrow \quad B_k= \lim_{\eta \to0} B_{k,\eta}.\]
	Then by the integral representation of Bessel function \cite[Page 48]{Watson1944} as follows
	\[J_{k+1/2} (r)= \frac{(r/2)^{k+1/2}}{k! \sqrt{\pi}} \int_{-1}^1 e^{ir t} (1-t^2)^k \ddd t,\]
	for each $\eta >0$, we can use the Fubini theorem to deduce
	\[B_{k,\eta} = \frac{1}{\pi (k!)^2 2^{2k+1}} \int_{-1}^1 \int_{-1}^1 (1-t^2)^k (1-s^2)^k K_{k,\eta}(t+s) \ddd t \ddd s,\]
	where
	\[K_{k,\eta}(t):= \int_0^{\infty} e^{-\eta r} r^{2k} \cos(2r) e^{itr} \ddd r.\]
	By the Euler integral formula
	\[\int_0^{\infty} r^k e^{-\alpha r} \ddd r= \frac{k!}{\alpha^{k+1}}, \quad \mathrm{Re} \alpha>0,\]
	we can write $\cos(2r)=(e^{2ix} + e^{-2ix})/2$ and then obtain
	\[K_{k,\eta} (t)= \frac{(2k)!}{2} \l[\frac{1}{(\eta -i(t+2))^{2k+1}} +\frac{1}{(\eta -i(t-2))^{2k+1}}\r],\]
	which satisfies
	\[\l|K_{k,\eta}(t)\r| \leq \frac{(2k)!}{2}\l(\frac{1}{|t+2|^{2k+1}} + \frac{1}{|2-t|^{2k+1}}\r).\]
	Hence we obtain the following estimate
	\[\l|(1-t^2)^k (1-s^2)^k K_{k,\eta}(t+s)\r| \lesssim (1-t^2)^k (1-s^2)^k \l(\frac{1}{|t+s+2|^{2k+1}} + \frac{1}{|2-t-s|^{2k+1}}\r) =: \widetilde{K}_k(t,s),\]
	and one can check that $\widetilde{K}_k(t,s) \in L^1([-1,1]^2)$. Therefore, we can use the dominated convergence theorem to conclude
	\[B_k = \lim_{\eta \to 0} B_{k,\eta} = \frac{1}{\pi (k!)^2 2^{2k+1}} \int_{-1}^1 \int_{-1}^1 (1-t^2)^k (1-s^2)^k K_{k}(t+s) \ddd t \ddd s,\]
	where
	\[K_k(t):= \lim_{\eta\to 0}K_{k,\eta}(t) = \frac{i(2k)! (-1)^k}{2} \l(\frac{1}{(t+2)^{2k+1}} - \frac{1}{(2-t)^{2k+1}}\r).\]
	Now, due to the fact $K_k(-t-s) = -K_k(t+s)$, we directly obtain
	\[B_k=0, \quad \Rightarrow \quad c_k= A_k +B_k = \frac{1}{(2k+1)\pi},\]
	which directly implies that the sequence $c_k$ is decreasing.
	
	\textbf{Step 2:} the spectral-gap constant. By using the function in \eqref{E:Sphere-spectral-gap-1} and the identity \eqref{E:Sphere-spectral-gap-2}, we can see that
	\[\frac{\delta_{*}^{''}(1)[f,f]}{\|f\|_2^2}\geq \frac{2\mathbf{M}^2 c_0^{-1} \sum_{k\geq 2} (c_0 -3c_k) |a_k|^2}{\sum_{k\geq 0}|a(k)|^2},\]
	where the equality holds if and only if $a(k) =0$ for odd number $k$ and $a(k)$ is real-valued for even number $k$. In the previous step we have shown that $c_k$ decreases with respect to $k\geq 2$. Hence the weighted inequality gives that
	\[\frac{\delta_{*}^{''}(1)[f,f]}{\|f\|_2^2} \geq 2\mathbf{M}^2 c_0^{-1} (c_0 -3c_2) = \frac{8\pi^2}{5},\]
	where the equality holds if we choose
	\[f(\theta)=a Y_2^0(\theta) = a \sqrt{\frac{5}{16\pi}} (3\cos^2\theta -1), \quad a\in \bR.\]
	In summary, we obtain the desired explicit value of sphere-spectral-gap constant $C_{**}= \frac{8\pi^2}{5}$.
	
	\textbf{Step 3:} stability constant v.s. spectral-gap constant. By taking the function
	\[f_{\varepsilon}(\theta) := 1+\varepsilon Y_2^0(\theta) = 1+\varepsilon \sqrt{\frac{5}{16\pi}} (3\cos^2\theta -1),\]
	we investigate the corresponding Rayleigh quotient
	\[E(f_{\varepsilon}):= \frac{\delta_{*}(f_{\varepsilon})}{\dist(f_{\varepsilon},\mathcal{C})^2}.\]
	First for the item $\|f_{\varepsilon}\|_{L^2(\bS^2)}$, by orthogonality, we directly obtain
	\begin{equation} \label{E:Sphere-spectral-gap-6}
		\|f_{\varepsilon}\|_{L^2(\bS^2)}^2 = |\bS^2| +\varepsilon^2.
	\end{equation}
	Second for the item $\dist(f,\mathcal{C})$, by direct computation, we investigate
	\begin{align*}
		m(f_{\varepsilon}) &= \frac{1}{4\pi} \sup_{x\in \bR^3} \l[\int_{\bS^2} (1+\varepsilon Y_2^0(\theta)) e^{-ix\theta} \ddd \sigma(\theta)\r]^2 \\
		&= 4\pi \sup_{x\in\bR^3} \l[\frac{\sin |x|}{|x|} -\varepsilon j_2(|x|)Y_2^0\l(\frac{x}{|x|}\r)\r]^2 \\
		&=: 4\pi \sup_{x\in \bR^3} H_{\varepsilon}(x).
	\end{align*}
	We are going to show that for each $\varepsilon \in [0,\frac{2(1-\sin 1) \sqrt{5\pi}}{35})$, there holds
	\begin{equation*}
		\sup_{x\in\bR^3} H_{\varepsilon}(x) = 1, \quad \Rightarrow \quad m(f_{\varepsilon}) = 4\pi,
	\end{equation*}
	which, by the constant distance representation Lemma \ref{L:Constant distance representation}, directly implies that
	\begin{equation} \label{E:Sphere-spectral-gap-7}
		\dist(f,\mathcal{C})=\varepsilon^2, \quad \text{for} \;\; \varepsilon \in \l[0,\frac{2(1-\sin 1) \sqrt{5\pi}}{35} \r).
	\end{equation}
	To show this fact, note that
	\[Y_2^0(\theta) \in \l[-\sqrt{5/(16\pi)}, \sqrt{5/(4\pi)}\r], \quad \Rightarrow \quad |Y_2^0|\leq \sqrt{5/(4\pi)}, \quad \forall \theta\in \bS^2.\]
	Then writing spherical coordinate $x=r\theta$, for any $r>0$, we obtain
	\[|\sup_{\theta \in \bS^2} H_{\varepsilon}(x)|\leq \l(\l|\frac{\sin r}{r}\r| + \varepsilon \sqrt{5/(4\pi)} |j_2(r)|\r)^2.\]
	Next, for this item, we divide the proof into two parts: the case $r\in[0,1]$ and the case $r\in [1,\infty)$. \\
	\emph{Case 1:} for $r\in[0,1]$, by the Taylor formula
	\[\sin r= r-\frac{r^3}{6} + \frac{r^5}{120} + R_1, |R_1|\leq \frac{r^7}{5040}; \quad \cos r =1-\frac{r^2}{2} + \frac{r^4}{24} + R_2, |R_2| \leq \frac{r^6}{720},\]
	recalling the relation \eqref{E:Sphere-spectral-gap-3}, we conclude
	\begin{align*}
		|j_2(r)| &= \l|\frac{r^2}{15} -\frac{r^4}{120} + \l(\frac{3}{r^3} + \frac{1}{r}\r) R_1 -\frac{3}{r^3} R^2\r| \\
		& \leq \frac{r^2}{15} +\frac{r^4}{120} + \l(\frac{3r^4}{5040} + \frac{r^6}{5040}\r) + \frac{3r^4}{720} \\
		&\leq \frac{r^2}{12}.
	\end{align*}
	And one can similarly obtain
	\[\l|\frac{\sin r}{r}\r| \leq 1-\frac{r^2}{7}, \quad \forall r\in[0,1].\]
	Hence, for $0\leq \varepsilon <\frac{24\sqrt{5\pi}}{35}$, we conclude
	\[|\sup_{\theta \in \bS^2} H_{\varepsilon}(x)|\leq \l[1- r^2 \l(\frac{1}{7} -\frac{\varepsilon \sqrt{5/(4\pi)} r^2}{12}\r)\r]^2 \leq 1, \quad \forall r\in [0,1],\]
	where the equality $|\sup_{\theta \in \bS^2} H_{\varepsilon}(x)| =1$ holds if and only if $r=0$. \\
	\emph{Case 2:} for $r\in[1, +\infty)$, it is not hard to see the following estimates
	\[|\sin r/r|\leq 1/r \leq 1, \quad |\sin r/r| \leq \sin 1,\]
	and
	\[|j_2(r)| \leq \frac{3}{r^3}+ \frac{3}{r^2} +\frac{1}{r} \leq 7.\]
	Hence, for $0\leq \varepsilon <\frac{2(1-\sin 1)\sqrt{5\pi}}{35}$, we conclude
	\[|\sup_{\theta \in \bS^2} H_{\varepsilon}(x)|\leq \sin 1 + 7\varepsilon \sqrt{5/(4\pi)} <1, \quad \forall r\in[1,+\infty).\]
	In summary, due to the fact $\frac{2(1-\sin 1)\sqrt{5\pi}}{35} < \frac{24\sqrt{5\pi}}{35}$, for each $\varepsilon \in [0,\frac{2(1-\sin 1)\sqrt{5\pi}}{35})$ we see that
	\[\sup_{x\in \bR^3} H_{\varepsilon}(x) =\lim_{|x|\to 0} H_{\varepsilon}(x)=1.\]
	Thus we obtain the desired conclusion \eqref{E:Sphere-spectral-gap-7}. Third for the item $\|\widehat{f_{\varepsilon}\sigma}\|_{L^4(\bR^3)}^4$, we recall that
	\[\widehat{\sigma}(x) = 4\pi \frac{\sin |x|}{|x|}, \quad \widehat{Y_2^0\sigma}(x) = -4\pi j_2(|x|) Y_2^0(x/|x|),\]
	and
	\[\int_{\bS^2} Y_2^0(\theta) \ddd \sigma(\theta) =0,\quad \int_{\bS^2} \l(Y_2^0(\theta)\r)^2 \ddd \sigma(\theta) =1, \quad \int_{\bS^2} \l(Y_2^0(\theta)\r)^3 \ddd \sigma(\theta) = \frac{\sqrt{5}}{7\sqrt{\pi}}.\]
	Thus we conclude
	\[\|\widehat{f_{\varepsilon}\sigma}\|_{L^4(\bR^3)}^4 = I_0 +I_1 \varepsilon +I_2 \varepsilon^2 + I_3\varepsilon^3 +I_4\varepsilon^4,\]
	where
	\[I_0:= \|\widehat{\sigma}\|_{L^4(\bR^3)}^4, \quad I_1:= 4\int_{\bR^3} |\widehat{\sigma}(x)|^2 \mathrm{Re} \l(\widehat{\sigma}(x) \overline{\widehat{Y_2^0 \sigma}}(x) \r) \ddd x,\]
	and
	\[I_2:= \int_{\bR^3} \l[2|\widehat{\sigma}|^2 |\widehat{Y_2^0 \sigma}(x)|^2 + 4\l(\mathrm{Re}\l(\widehat{\sigma}(x) \overline{\widehat{Y_2^0 \sigma}}(x) \r)\r)^2\r] \ddd x = 6 \int_{\bR^3} (\widehat{\sigma})^2 \l(\widehat{Y_2^0 \sigma}(x)\r)^2 \ddd x,\]
	as well as
	\[I_3:= 4\int_{\bR^3} |\widehat{Y_2^0 \sigma}(x)|^2 \mathrm{Re}\l(\widehat{\sigma}(x) \overline{\widehat{Y_2^0 \sigma}}(x) \r) \ddd x = 4\int_{\bR^3} \l(\widehat{Y_2^0 \sigma}(x)\r)^3 \widehat{\sigma}(x) \ddd x, \quad I_4:= \|\widehat{Y_2^0 \sigma}\|_{L^4(\bR^3)}^4.\]
	Writing spherical coordinate $x=r\theta \in [0,\infty) \times \bS^2$, we can directly compute that
	\[I_0=4^5 \pi^5 \int_0^{\infty} \frac{\sin^4 r}{r^2} \ddd r, \quad I_1 = -4^5\pi^4 \int_0^{\infty} \frac{\sin^3 r}{r} j_2(r) \ddd r \int_{\bS^2} Y_2^0(\theta) \ddd \sigma(\theta) =0,\]
	and
	\[I_2= 6(4\pi)^4 \int_0^{\infty} \sin^2 r j_2^2(r) \ddd r, \quad I_3= -4^5 \pi^4 \int_0^{\infty} r\sin r j_2^3 (r) \ddd r \int_{\bS^2} \l(Y_2^0(\theta)\r)^3 \ddd \sigma(\theta).\]
	Then, similar to Step 1, by some direct computation\footnote{One can also use some software such as \textit{Mathematica} to check these integrals.}, one can obtain
	\[\int_0^{\infty} \frac{\sin^4 r}{r^2} \ddd r = \frac{\pi}{4}, \quad \int_0^{\infty} \sin^2 r j_2^2(r) \ddd r = \frac{\pi}{20}, \quad \int_0^{\infty} r\sin r j_2^3 (r) \ddd r = -\frac{\pi}{28},\]
	these integrals further give the following
	\[I_0=256\pi^6, \quad I_1= 0, \quad I_2 = \frac{384\pi^5}{5}, \quad I_3 = \frac{256\sqrt{5} \pi^{9/2}}{49}.\]
	In summary, for small $\varepsilon \ll 1$, we have established the following three conclusions
	\begin{itemize}
		\item[\ding{172}] $\|f_{\varepsilon}\|_{L^2(\bS^2)}^2 = |\bS^2| +\varepsilon^2$.
		\item[\ding{173}] $\|\widehat{f_{\varepsilon} \sigma}\|_{L^4(\bR^3)}^4 = 256\pi^6 +\frac{384\pi^5}{5} \varepsilon^2 + \frac{256\sqrt{5} \pi^{9/2}}{49} \varepsilon^3 +o(\varepsilon^3)$.
		\item[\ding{174}] $\dist(f_{\varepsilon},\mathcal{C})^2 = \varepsilon^2$.
	\end{itemize}
	Therefore, as $\varepsilon\to0$, we obtain that
	\begin{align*}
		E(f_{\varepsilon}) &= \frac{16\pi^3 +4\pi^2 \varepsilon^2 -\l(256\pi^6 +\frac{384\pi^5}{5} \varepsilon^2 + \frac{256\sqrt{5} \pi^{9/2}}{49} \varepsilon^3 +o(\varepsilon^3)\r)^{1/2}}{\varepsilon^2} \\
		&= \frac{8\pi^2}{5} -\frac{8\sqrt{5} \pi^{3/2}}{49} \varepsilon +o(\varepsilon).
	\end{align*}
	This gives the desired result $C_{**}< C_{SG*}= 8\pi^2/5$, and thus we have finished the proof.
\end{proof}

\subsection{Two-peak}
In this subsection, we compute the two-peak constant $C_{TP*}$ in Proposition \ref{P:Sphere-two-peak}, which essentially comes from the profile decomposition conclusion, and this classical decomposition can be seen in the works such as \cite{CS2012A&P,FLS2016,FS2024}.
\begin{lemma}[Constant distance representation] \label{L:Constant distance representation}
	The constant distance functional satisfies
	\[\dist(f,\mathcal{C})^2 = \|f\|_2^2 - m(f), \quad m(f):= \sup_{g\in \mathcal{C}_1} \l(\Re\int_{\bS^2} f \bar{g} \ddd \sigma(\theta)\r)^2, \quad \mathcal{C}_1:= \l\{g\in \mathcal{C}: \|g\|_2=1\r\}.\]
	And for each function $f\in L^2(\bS^2)$, its constant distance $\dist(f,\mathcal{C})$ can be attained.
\end{lemma}
\begin{proof}[\textbf{Proof of Lemma \ref{L:Constant distance representation}}]
	The first part of this lemma follows from a direct expansion, and the second part can be shown by applying an approximate identity argument \cite[Section 1.2.4]{Grafakos2014a}. For the sake of brevity, we omit the detailed proof here. Further details can be found in \cite[Lemma 2.2]{Konig2025} and \cite[Lemma 3.3]{DEFFL2025}.
\end{proof}
\begin{proof}[\textbf{Proof of Proposition \ref{P:Sphere-two-peak}}]
	First, by the sphere profile decomposition as shown in \cite{CS2012A&P,FLS2016,FS2024}, one can see that
	\begin{equation} \label{E:Sphere-two-peak-1}
		\|f_y\|_{L^2(\bS^2)}^2 = 2\|1\|_{L^2(\bS^2)}^2 +o_{|y|\to \infty}(1) = 8\pi +  o_{|y|\to \infty}(1),
	\end{equation}
	and
	\begin{equation} \label{E:Sphere-two-peak-2}
		\|\widehat{f_y \sigma}\|_{L^4(\bR^3)}^4 = 2 \mathbf{M}^4 \|1\|_{L^2(\bS^2)}^4 +o_{|y|\to \infty}(1) = 512 \pi^6 + o_{|y|\to \infty}(1).
	\end{equation}
	For the item $\dist(f,\mathcal{C})^2$, by the constant distance representation Lemma \ref{L:Constant distance representation}, we investigate
	\[m(f_y)= \sup_{x\in \bR^3} \l(\int_{\bS^2} f_y(\theta) e^{ix\theta} \ddd \sigma(\theta)\r)^2 \Big/ \|e^{ix\theta}\|_{L^2(\bS^2)}^2 =\frac{1}{4\pi} \sup_{x\in \bR^3} \l(\int_{\bS^2} \l(1+e^{iy\theta} \r) e^{ix\theta} \ddd \sigma(\theta)\r)^2.\]
	By a direct computation, we conclude
	\[m(f_y)= 4\pi \sup_{x\in \bR^3} \l(\frac{\sin|x|}{|x|} + \frac{\sin|x+y|}{|x+y|}\r)^2.\]
	Then by an elementary analysis, one can obtain
	\[\lim_{|y|\to \infty} m(f_y) = 4\pi.\]
	Hence, due to Lemma \ref{L:Constant distance representation}, we obtain the limit behavior estimate
	\begin{equation} \label{E:Sphere-two-peak-3}
		\dist(f_y,\mathcal{C})^2 = 4\pi +o_{|y|\to \infty}(1).
	\end{equation}
	Combining the estimates \eqref{E:Sphere-two-peak-1} and \eqref{E:Sphere-two-peak-2} and \eqref{E:Sphere-two-peak-3}, one can see that
	\[\frac{\mathbf{M}^2 \l\|f_y \r\|_{L^2(\bS^2)}^2 -\l\|\widehat{f_y \sigma} \r\|_{L^4(\bR^3)}^2}{\dist(f_y,\mathcal{C})^2} = (2-\sqrt{2})4\pi^2 +o_{|y|\to \infty}(1).\]
	This gives the desired result and completes the proof.
\end{proof}

\subsection{Compactness of minimizing sequences}
In this subsection, based on the spectral-gap result Proposition \ref{P:Sphere-spectral-gap} and two-peak vanishing result Proposition \ref{P:Sphere-two-peak}, we show the desired existence of minimizers for the two-dimensional sphere stability constant $C_{**}$.
\begin{proof}[\textbf{Proof of Theorem \ref{T:Stability minimizer-2D sphere}}]
	The strategy is similar to the \cite[Section 4, Proof of Theorem 1.2]{Konig2025} and the Proof of Theorem \ref{T:Stability minimizer-paraboloid} in the previous Section \ref{SubS:Conditional compactness-paraboloid}. For simplicity, we only show the main ideas and omit the details. In this two-dimensional sphere case, we can obtain the unconditional result since we can establish the strict inequality in Proposition \ref{P:Sphere-spectral-gap}. Therefore, by following the outline in the Proof of Theorem \ref{T:Stability minimizer-paraboloid}, it remains only to check that
	\[\mathbf{M}^2 -C_{SG*} = 4\pi^2 -\frac{8\pi^2}{5} =\frac{12\pi^2}{5}>0,\]
	which essentially can give the nonzero weak limit of minimizing sequences
	\[f_n \rightharpoonup f_0 \;\; \text{in} \;\; L^2(\bS^2), \quad f_0\neq 0.\]
	Then the desired existence of minimizers follows from standard arguments.
\end{proof}

\bigskip\bigskip
\newcommand{\etalchar}[1]{$^{#1}$}

\end{document}